\documentclass[11pt]{amsart}
\usepackage{amsmath,amsthm,mathrsfs,amsfonts,amssymb,color}

\newtheorem{thm}{Theorem}[section]
\newtheorem{lem}[thm]{Lemma}

\theoremstyle{definition}

\newtheorem{dfn}[thm]{Definition}

\DeclareMathOperator{\diam}{diam}

\newcommand{\PP}{\mathbb{P}}

\newcommand{\GGp}{{\mathop G\limits^{\circ}}}

\newcommand{\GG}{\mathop G \limits^{    \circ}}

\newcommand{\bmo}{{\rm BMO}}
\newcommand{\vmo}{{\rm VMO}}
\newcommand{\cmo}{{\rm CMO}}
\newcommand{\littlebmo}{\textit{bmo}}
\newcommand{\Z}{\mathbb{Z}}
\newcommand{\D}{\mathcal{D}}
\newcommand{\intav}{-\hspace{-0.05cm}\!\!\!\!\!\!\int}

\newcommand{\om}{\omega}
\newcommand{\Om}{\Omega}

\numberwithin{equation}{section}
\allowdisplaybreaks
\begin{document}

\baselineskip 16pt
\hfuzz=6pt

\title[Weak-star convergence in product $H^1$ on spaces of
homogeneous type]{On weak-star convergence in product Hardy
spaces on spaces of homogeneous type}

\author[Ming-Yi Lee, Ji Li and Lesley A.~Ward]{Ming-Yi Lee,
Ji Li and Lesley A.~Ward}

\address{Department of Mathematics\\
National Central University\\
Chung-Li 320, Taiwan \newline
\& \newline
National Center for Theoretical Sciences\\
1 Roosevelt Road, Sec. 4\\
National Taiwan University\\
Taipei 106\\
Taiwan}
\email{mylee@math.ncu.edu.tw}

\address{Department of Mathematics\\
Macquarie University\\
NSW, 2109, Australia}
\email{ji.li@mq.edu.au}

\address{
School of Information Technology and Mathematical
Sciences, University of South Australia, Mawson Lakes SA 5095,
Australia}
\email{{\tt lesley.ward@unisa.edu.au}}

\thanks{
The first author is supported by Grant \#MOST
103-2115-M-008-001. This article was written while the second
author was visiting National Central University. The second
author would like to thank the Mathematics Research Promotion
Center for support. The second author is supported
by the Australian Research Council, grant no.~ARC-DP160100153, and the third author is supported
by the Australian Research Council, grant no.~ARC-DP120100399.
}



\begin{abstract}
    A classical theorem of Jones and Journ\'e on weak-star
    convergence in the Hardy space~$H^1$ was generalised to the
    multiparameter setting by Pipher and Treil. We prove the
    analogous result when the underlying space is a product
    space of homogeneous type. The main tools we use for this
    setting are from recent work in papers by Chen, Li and Ward
    and by Han, Li and Ward.
\end{abstract}

\maketitle

\section{Introduction}\label{sec-Intro}
\setcounter{equation}{0}

In this paper we extend to the setting of product Hardy
spaces~$H^1$ on spaces of homogeneous type the result that
almost-everywhere convergence of a sequence of uniformly
bounded $H^1$ functions implies weak-star convergence.
See~\cite{PT} for the history of this result and its
connections with commutators, singular integral operators,
Riesz transforms, $\bmo$, div-curl lemmas, and the theory of
compensated compactness in partial differential equations.

Our main result is the following.

\begin{thm}\label{thm:1.1}
    Suppose that a sequence of functions $\{f_k\}\subset
    H^1(X_1\times \cdots\times X_n)$ satisfies
    $\|f_k\|_{H^1}\le 1$ for all $k$ and $f_k(x)\to f(x)$ for
    $\mu$-almost every $x\in X_1\times \cdots\times X_n$. Then
    $f\in H^1(X_1\times \cdots\times X_n)$,  $\|f\|_{H^1}\le
    1$, and for all $\phi\in \vmo(X_1\times \cdots\times X_n)$,
    \begin{equation}\label{eq:1.1}
        \int_{X_1\times \cdots\times X_n} f_k(x)\phi(x) \, d\mu(x)
        \longrightarrow \int_{X_1\times \cdots\times X_n} f(x)\phi(x) \, d\mu(x).
    \end{equation}
\end{thm}

To extend the Calder\'on--Zygmund singular integral operator
theory to a more general setting, in the early 1970s Coifman
and Weiss introduced spaces of homogeneous type. As Meyer
remarked in his preface to~\cite{DH}, \emph{``One is amazed by
the dramatic changes that occurred in analysis during the
twentieth century. \ldots
After many improvements, mostly achieved by the
Calder\'on--Zygmund school, the action takes place today on
spaces of homogeneous type. No group structure is available,
the Fourier transform is missing, but a version of harmonic
analysis is still present. Indeed the geometry is conducting
the analysis.''} We say that $(X,d,\mu)$ is a \emph{space of
homogeneous type in the sense of Coifman and Weiss} if $d$ is a
quasi-metric on~$X$ and $\mu$ is a nonzero measure satisfying
the doubling condition. To be more precise, let us begin by
recalling these spaces. A quasi-metric $d$ on a set~$X$ is a
function $d:X\times X\longrightarrow[0,\infty)$ satisfying
\begin{enumerate}
    \item[(1)] $d(x,y) = d(y,x) \geq 0$ for all $x$, $y\in
        X$;

    \item[(2)] $d(x,y) = 0$ if and only if $x = y$; and

    \item[(3)] the quasi-triangle inequality holds: there
        exists a constant $A_0\in [1,\infty)$ such that for
        all $x$, $y$ and $z\in X$,
    \begin{eqnarray}\label{quasi metric}
        d(x,y)
        \leq A_0 [d(x,z) + d(z,y)].
    \end{eqnarray}
\end{enumerate}
We define the quasi-metric ball by $ B(x,r) := \{y\in X: d(x,y)
< r\}$ for $x\in X$ and $r > 0$. Note that the quasi-metric, in
contrast to a metric, may not be H\"older regular and
quasi-metric balls may not be open. In this paper, we assume that
\begin{enumerate}
    \item[(4)] given a neighborhood $N$ of a point $x$
        there is an $\epsilon>0$ such that the sphere
        $\{y\in X: d(x,y)\le \epsilon\}$ with center at $x$
        is contained in $N$; and

    \item[(5)] the sphere $\{y\in X:  d(x,y)\le r\}$ is
        measurable, and the measure $\mu(\{y\in X:
        d(x,y)\le r\})$ is a continuous function of $r$ for
        each $x$.
\end{enumerate}
We say that a nonzero measure $\mu$ satisfies the
\emph{doubling condition} if there is a constant $C_\mu$ such
that for all $x\in X$ and all $r > 0$,
\begin{eqnarray}\label{doubling condition}
   \mu(B(x,2r))
   \leq C_\mu \mu(B(x,r))
   < \infty.
\end{eqnarray}

As noted by the reviewer of~\cite{PT} in Mathematical Reviews,
since $H^1$ is not reflexive, the fact that $H^1$ is the dual
of $\vmo$ does not lead to a functional analytic proof of this
result using known methods.

The paper is organised as follows. In Section~\ref{sec-Pre} we
present some background about spaces of homogeneous type. In
Section~\ref{sec-thm1.1-one-parameter} we prove the
one-parameter version of our result, and in
Section~\ref{sec-thm1.1-product} we prove the product version.

\section{Preliminaries}\label{sec-Pre}
\setcounter{equation}{0}

We recall the ingredients and tools that we will use below to
prove Theorem~\ref{thm:1.1}, namely systems of dyadic cubes,
the orthonormal basis and wavelet expansion of Auscher and
Hyt\"onen~\cite{AH}, the spaces of test functions and of
distributions, the definitions from~\cite{HLW} (using these
spaces) of $H^1$, $\bmo$ and $\vmo$ on product spaces of
homogeneous type, and the duality relations between them.
See~\cite{HLW} for a full account of this material.

\subsection{Systems of dyadic cubes in a doubling quasi-metric space}

Let $X$ be a set equipped with a quasi-metric $d$ and a
doubling measure $\mu$; in particular, $(X,d,\mu)$ is a
space of homogeneous type. As shown in~\cite{HK}, building
on~\cite{Chr}, there exists a dyadic decomposition for this
space~$X$. There exist positive absolute constants $c_1$, $C_1$
and $0<\delta<1$ such that we can construct a set of points
$\{x_\alpha^k\}_{k,\alpha}$ and families of sets
$\{Q_\alpha^k\}_{k,\alpha}$ in $X$ satisfying the following
properties:
\begin{eqnarray}
    &&\textup {if } \ell\leq k, \textup { then either } Q_\alpha^k\subset
        Q_\beta^\ell \textup { or } Q_\alpha^k
        \cap Q_\beta^\ell=\emptyset;\label {DyadicP1}\\
    &&\textup {for every } k\in \Z \textup { and } \alpha\neq\beta,  Q_\alpha^k\cap
       Q_\beta^k=\emptyset; \label {DyadicP2}\\
    &&\textup {for every } k\in \Z,\,
        X=\bigcup_\alpha Q_\alpha^k;\label {DyadicP3} \\
    && B(x_\alpha^k,c_1\delta^k)\subset Q_\alpha^k\subset
        B(x_\alpha^k,C_1\delta^k);\label {prop_cube4}\\
    &&\textup {if } \ell\leq k \textup { and } Q_\alpha^k\subset Q_\beta^\ell,
        \textup { then } B(x_\alpha^k,C_1\delta^k)\subset
        B(x_\beta^\ell,C_1\delta^\ell).\label {DyadicP5}
\end{eqnarray}
Here for each $k\in\mathbb{Z}$, $\alpha$ runs over an
appropriate index set. We call the set $Q_\alpha^k$ a
\emph{dyadic cube} and $x_\alpha^k$ the \emph{center} of the
cube. Also, $k$ is called the \emph{level} of this cube. We
denote the collection of dyadic cubes at level $k$ by $\D^k$,
and the collection of all dyadic cubes by~$\D$. When
$Q_\alpha^k\subset Q_\beta^{k-1}$, we say $Q_\alpha^k$ is a
child of $Q_\beta^{k-1}$ and $Q_\beta^{k-1}$ is the parent
of~$Q_\alpha^{k}$. Because $X$ is a space of homogeneous type,
there is a uniform constant~$\mathcal{N}$ such that each cube
$Q\in\D$ has at most~$\mathcal{N}$ children.

\subsection{Orthonormal basis and wavelet expansion}


We recall the orthonormal basis and wavelet expansion of
$L^2(X)$ due to Auscher and Hyt\"onen~\cite{AH}. To state their
result, we first recall the set of {\it reference dyadic
points} $x_\alpha^k$ as follows. Let $\delta$ be a fixed small
positive parameter (for example, as pointed out in Section 2.2
of \cite{AH}, it suffices to take $\delta\leq 10^{-3}
A_0^{-10}$). For $k=0$, let
$\mathscr{X}^0:=\{x_\alpha^0\}_\alpha$ be a maximal collection
of 1-separated points in~$ X $. Inductively, for
$k\in\mathbb{Z}_+$, let $\mathscr{X}^k:=\{x_\alpha^k\}
\supseteq \mathscr{X}^{k-1}$ and
$\mathscr{X}^{-k}:=\{x_\alpha^{-k}\} \subseteq
\mathscr{X}^{-(k-1)}$ be $\delta^k$- and
$\delta^{-k}$-separated collections in $\mathscr{X}^{k-1}$ and
$\mathscr{X}^{-(k-1)}$, respectively.

Lemma 2.1 in \cite{AH} shows that, for all $k\in\mathbb{Z}$ and
$x\in X$, the reference dyadic points satisfy
\begin{eqnarray}\label{delta sparse}
    d(x_\alpha^k,x_\beta^k)\geq\delta^k\ (\alpha\not=\beta),\hskip1cm
        d(x,\mathscr{X}^k)
    =\min_\alpha\,d(x,x_\alpha^k)<2A_0\delta^k.
\end{eqnarray}

Now let $c_0:=1$, $C_0:=2A_0$ and $\delta\leq 10^{-3} A_0^{-10}$.
Then
%
%
there exists a set of half-open dyadic cubes
$\{Q_\alpha^k\}_{k\in\mathbb{Z},\alpha\in\mathscr{X}^k}$
associated with the reference dyadic points
$\{x_\alpha^k\}_{k\in\mathbb{Z},\alpha\in\mathscr{X}^k}$. We
consider the reference dyadic point $x_\alpha^k$ as the
\emph{center} of the dyadic cube $Q_\alpha^k$. We also identify
with $\mathscr{X}^k$ the set of indices $\alpha$ corresponding
to $x_\alpha^k \in \mathscr{X}^{k}$.

Note that $\mathscr{X}^{k}\subseteq \mathscr{X}^{k+1}$ for
$k\in\mathbb{Z}$, so that every $x_\alpha^k$ is also a point of the
form $x_\beta^{k+1}$, and thus  of all the finer levels. We denote
$\mathscr{Y}^{k}:=\mathscr{X}^{k+1}\backslash \mathscr{X}^{k}$, and
relabel the points $\{x_\alpha^k\}_\alpha$ that belong
to~$\mathscr{Y}^k$ as $\{y_\alpha^k\}_\alpha$.


\begin{thm}[\cite{AH} Theorem 7.1]\label{theorem AH orth basis}
    Let $(X,d,\mu)$ be a space of homogeneous type with
    quasi-triangle constant $A_0$, and let $a:=(1+2\log_2A_0)^{-1}$. There
    exists an orthonormal basis $\psi_\alpha^k$, $k\in\mathbb{Z}$,
    $y_\alpha^k\in \mathscr{Y}^k$, of $L^2(X)$, having exponential decay
    \begin{eqnarray}\label{exponential decay}
        |\psi_\alpha^k(x)|
        \leq {C\over \sqrt{\mu(B(y_\alpha^k,\delta^k))}}
            \exp(-\nu(\delta^{-k} d(y_\alpha^k,x))^a),
    \end{eqnarray}
    H\"older-regularity
    \begin{eqnarray}\label{Holder-regularity}
        |\psi_\alpha^k(x)-\psi_\alpha^k(y)|
        \leq {C\over \sqrt{\mu(B(y_\alpha^k,\delta^k))}}
            \Big( {d(x,y)\over\delta^k}\Big)^\eta
            \exp(-\nu(\delta^{-k} d(y_\alpha^k,x))^a)
    \end{eqnarray}
    for some $\eta\in(0,1]$ and for $d(x,y)\leq \delta^k$, and the
    cancellation property
    \begin{eqnarray}\label{cancellation}
        \int_X \psi_\alpha^k(x) \, d\mu(x) = 0,\
            k\in\mathbb{Z},\
            y_\alpha^k\in\mathscr{Y}^k.
    \end{eqnarray}
    Moreover,
    \begin{eqnarray}\label{reproducing formula}
        f(x)
        =\sum_{k\in\mathbb{Z}}\sum_{\alpha \in \mathscr{Y}^k}
            \langle f,\psi_{\alpha}^k \rangle \psi_{\alpha}^k(x)
    \end{eqnarray}
    in the sense of $L^2(X)$.

    Here $\delta$ is a fixed small parameter, say $\delta\leq
    {1\over 1000}A_0^{-10}$, and $\nu > 0$ and $C<\infty$ are
    two constants that are independent of $k$, $\alpha$, $x$
    and $y_\alpha^k$. In what follows, we also refer to the
    functions $\psi_\alpha^k$ as wavelets.
\end{thm}

\subsection{Spaces of test functions and distributions}

We refer the reader to~\cite{HLW}, Definitions~3.9 and~3.10 and
the surrounding discussion, for the definitions of the
space~$\GG$ of product test functions and its dual space
$(\GG)'$ of product distributions on the product space
$X_1\times X_2$. In \cite{HLW}, $\GG$ is denoted by
$\GG(\beta_1,\beta_2;\gamma_1,\gamma_2)$ and $(\GG)'$ is
denoted by $\GG(\beta_1,\beta_2;\gamma_1,\gamma_2)'$, where the
$\beta_i$ and $\gamma_i$ are parameters that quantify the size
and smoothness of the test functions, and $\beta_i \in
(0,\eta_i)$ where $\eta_i$ is the regularity exponent from
Theorem~\ref{theorem AH orth basis}. (In fact, in~\cite{HLW}
the theory is developed for $\beta_i \in (0,\eta_i]$, but for
simplicity here we only use $\beta_i \in (0,\eta_i)$ since that
is all we need.) We note that the one-parameter scaled
Auscher--Hyt\"onen wavelets $\psi_\alpha^k(x)/
\sqrt{\mu(B(y_\alpha^k,\delta^k))}$ are test functions, and
that their tensor products $\psi^{k_1}_{\alpha_1}(x)
\psi^{k_2}_{\alpha_2}(y) \big(\mu_1(B(y^{k_1}_{\alpha_1},
\delta_1^{k_1})) \mu_2(B(y^{k_2}_{\alpha_2},
\delta_2^{k_2}))\big)^{-1/2}$ are product test functions
in~$\GG$, for all $\beta_i \in (0,\eta_i]$ and all $\gamma_i >
0$, for $i = 1$,~2. These facts follow from the theory
in~\cite{HLW}, specifically Definition~3.1 and the discussion
after it, Theorem~3.3, and Definitions 3.9 and~3.10 and the
discussion between them.

We have the following version of the reproducing formula in the
product setting $X_1\times X_2$.

\begin{thm}[\cite{HLW}]\label{thm product reproducing formula test function}
    The reproducing formula
    \begin{eqnarray}\label{product reproducing formula}
       f(x_1,x_2)
       = \sum_{k_1}\sum_{\alpha_1 \in \mathscr{Y}^{k_1}}\sum_{k_2}\sum_{\alpha_2 \in \mathscr{Y}^{k_2}}
        \langle f,\psi_{\alpha_1}^{k_1}\psi_{\alpha_2}^{k_2} \rangle
        \psi_{\alpha_1}^{k_1}(x_1)\psi_{\alpha_2}^{k_2}(x_2)
    \end{eqnarray}
    holds in both
    $\GGp(\beta_{1},\beta_{2};\gamma_{1},\gamma_{2})$ and
    $(\GGp(\beta_{1},\beta_{2};\gamma_{1},\gamma_{2}))'$, with
    $0<\beta_i <\eta_i$ and $\gamma_i<\eta_i$ for $i=1$, $2$.
\end{thm}

We recall from~\cite{HLW} the definitions of the Hardy
space~$H^1(X_1\times X_2)$, the bounded mean oscillation
space~$\bmo(X_1\times X_2)$, and the vanishing mean oscillation
space~$\vmo(X_1\times X_2)$.


\begin{dfn}[\cite{HLW}]\label{def-Hp}
The {\it product Hardy space} $H^1$ is defined by
\[
    H^1(X_1\times X_2)
    :=\big\lbrace f \in (\GG)':
        S(f)\in L^1( X _1\times X _2)\big\rbrace,
\]
where $S(f)$ is the product Littlewood--Paley square function defined as
\begin{eqnarray}
    S(f)(x_1,x_2)
    :=\Big\{ \sum_{k_1}\sum_{\alpha_1\in\mathscr{Y}^{k_1}}
        \sum_{k_2}\sum_{\alpha_2\in\mathscr{Y}^{k_2}}
        \big| \langle \psi_{\alpha_1}^{k_1}\psi_{\alpha_2}^{k_2},f \rangle
        \widetilde{\chi}_{{Q}_{\alpha_1}^{k_1}}(x_1)
        \widetilde{\chi}_{{Q}_{\alpha_2}^{k_2}}(x_2)
        \big|^2 \Big\}^{1/2},
\end{eqnarray}
where $\widetilde{\chi}_{{ Q}_{\alpha_i}^{k_i}}(x_i):=\chi_{{
Q}_{\alpha_i}^{k_i}}(x_i)\mu_i({ Q}_{\alpha_i}^{k_i})^{-1/2}$
and $\chi_{{ Q}_{\alpha_i}^{k_i}}(x_i)$ is the indicator
function of the dyadic cube ${ Q}_{\alpha_i}^{k_i}$ for
$i=1,2$.

For $f\in H^1(X_1\times X_2)$, we define $\|f\|_{H^1(X_1\times
X_2)}:=\|S(f)\|_{L^1(X_1\times X_2)}$.
\end{dfn}


\begin{dfn}[\cite{HLW}]\label{def-CMO}
We define the \emph{product $\bmo$ space} as
$$
\bmo( X _1\times X _2):=\big\{ f \in (\GG)':
\mathcal{C}_1(f)< L^\infty\},
$$
with $\mathcal{C}_1(f)$ defined as follows:
\begin{eqnarray}\label{Carleson norm}
\mathcal{C}_1(f):=\sup_{\Omega}\Big\{ {1\over\mu(\Omega)}\sum_{k_1,k_2\in\mathbb{Z}, \alpha_1\in\mathscr{Y}^{k_1},
\alpha_2\in\mathscr{Y}^{k_2},R=Q_{\alpha_1}^{k_1}\times
Q_{\alpha_2}^{k_2}\subset \Omega} \big| \langle
\psi_{\alpha_1}^{k_1}\psi_{\alpha_2}^{k_2},f \rangle
 \big|^2 \Big\}^{1/2},
\end{eqnarray}
where $\Omega$ runs over all open sets in $ X _1\times X _2$ with
finite measures.
\end{dfn}

Now we introduce the following
\begin{dfn}[\cite{HLW}]\label{def-vmo}
We define the \emph{product vanishing mean oscillation space}
$\vmo(X_1\times X_2)$ as the subspace of $\bmo(X_1\times X_2)$
consisting of those $f\in \bmo(X_1\times X_2)$
 satisfying the three properties
\begin{eqnarray*}
&&{\rm (a)}\hskip.5cm \lim_{\delta\rightarrow 0}\ \sup_{\Omega:\
\mu(\Omega)<\delta}\Big\{
{1\over\mu(\Omega)}\sum_{k_1,k_2\in\mathbb{Z},
\alpha_1\in\mathscr{Y}^{k_1},
\alpha_2\in\mathscr{Y}^{k_2},R=Q_{\alpha_1}^{k_1}\times
Q_{\alpha_2}^{k_2}\subset \Omega} \big| \langle
\psi_{\alpha_1}^{k_1}\psi_{\alpha_2}^{k_2},f \rangle
 \big|^2 \Big\}^{1/2}=0;\\
&&{\rm (b)}\hskip.5cm \lim_{N\rightarrow \infty}\ \sup_{\Omega:\
\diam(\Omega)>N}\Big\{
{1\over\mu(\Omega)}\sum_{k_1,k_2\in\mathbb{Z},
\alpha_1\in\mathscr{Y}^{k_1},
\alpha_2\in\mathscr{Y}^{k_2},R=Q_{\alpha_1}^{k_1}\times
Q_{\alpha_2}^{k_2}\subset \Omega} \big| \langle
\psi_{\alpha_1}^{k_1}\psi_{\alpha_2}^{k_2},f \rangle
 \big|^2 \Big\}^{1/2}=0;\\
&&{\rm (c)}\hskip.5cm \lim_{N\rightarrow \infty}\ \sup_{\Omega:\
\Omega \subset ( B(x_1,N)\times B(x_2,N))^c }\\
&&\hskip3cm\Big\{ {1\over\mu(\Omega)}\sum_{k_1,k_2\in\mathbb{Z},
\alpha_1\in\mathscr{Y}^{k_1},
\alpha_2\in\mathscr{Y}^{k_2},R=Q_{\alpha_1}^{k_1}\times
Q_{\alpha_2}^{k_2}\subset \Omega} \big| \langle
\psi_{\alpha_1}^{k_1}\psi_{\alpha_2}^{k_2},f \rangle
 \big|^2 \Big\}^{1/2}=0,\ {\rm where\ } \\
&&\hskip1.15cm x_1\ {\rm and}\ x_2\ {\rm are\ any\ fixed\ points\
in\ } X_1\ {\rm and}\ X_2 {\rm ,\ respectively.}
\end{eqnarray*}
\end{dfn}


\begin{thm}[\cite{HLW}]\label{thm-duality 2}
    The following duality results hold:
    \begin{align*}
        \big(H^1( X _1\times X _2)\big)'
        &= \bmo( X _1\times X _2), \\
        \big(\vmo( X _1\times X _2)\big)'
        &=  H^1( X _1\times X _2). \\
    \end{align*}
\end{thm}

\section{Proof of Theorem \ref{thm:1.1} for one parameter}\label{sec-thm1.1-one-parameter}
\setcounter{equation}{0}

To prove Theorem~\ref{thm:1.1}, paralleling the Euclidean
one-parameter case, we will make use of several properties of
the $A_p$ classes on spaces of homogeneous type. These
properties are collected in Lemma~\ref{ap}, Theorem~\ref{w1},
and Lemmas~\ref{a1}--\ref{a2} below.

Let $(X,d,\mu)$ be a space of homogeneous type. A nonnegative
locally integrable function $\omega : X \to \mathbb{R}$ is said
to belong to $A_p(X)$, $1<p<\infty$, if
\[
    \sup_{B}\bigg(\frac 1{\mu(B)}\int_{B} \omega(x) \, d\mu(x)\bigg)
         \bigg(\frac 1{\mu(B)}\int_{B} \omega(x)^{-\frac 1{p-1}} \, d\mu(x)\bigg)^{p-1}
    < \infty,
\]
and $\omega$ is said to belong to $A_1(X)$ if
\[
    \sup_{B}\bigg(\frac 1{\mu(B)}\int_{B} \omega(x) \, d\mu(x)\bigg)
    \Big(\operatornamewithlimits{ess\,sup}_{x\in B} \omega^{-1}(x) \Big)
    < \infty.
\]

\begin{lem}[\cite{C2} Lemma 4]\label{ap}
    Let $\omega\in A_p$, $1\le p<\infty$. There exists a
    constant $C>0$ such that, for any subset $E$ of $B$,
    \[
        \bigg(\frac{\mu(E)}{\mu(B)}\bigg)^p
        \le C\,\frac{\int_E \omega(x) \, d\mu(x)}{\int_B \omega(x) \, d\mu(x)}\,.
    \]
\end{lem}

The centered Hardy--Littlewood maximal operator $M$ with
respect to the measure $\mu$ is defined by
$$Mf(x) := \sup_{r>0}\frac 1{\mu(B(x,r))}\int_{B(x,r)}|f(y)|\,d\mu(y).$$

\begin{thm}[\cite{C2} Theorem 3]\label{w1}
    If $\omega\in A_1$, then $M$ is of $\omega$-weak type
    $(1,1)$ with respect to $\mu$; that is, there exists a
    constant $C>0$ such that, for all $\lambda>0$ and all $f\in
    L^1_\omega(d\mu)$,
    \[
        \int_{\{x\in X: M f(x) > \lambda\}} \omega(x) \, d\mu(x)
        \le \frac C\lambda\int_X |f(x)|\omega(x) \, d\mu(x).
    \]
\end{thm}

Similarly, the uncentered Hardy--Littlewood maximal operator
$M$ with respect to the measure~$\mu$ is defined by
$$\widetilde Mf(x) := \sup_{x\in B}\frac 1{\mu(B)}\int_B|f(y)|\,d\mu(y).$$

\begin{lem}\label{a1}
    The weight $\omega\in A_1$ if and only if there is a constant $C>0$ such that
    $$M\omega(x)\le C\omega(x)\qquad \mu\text{-almost everywhere}\ x\in X. $$
\end{lem}

\begin{proof}
Suppose that there is a constant $C>0$ such that $\widetilde
M\omega(x)\le C\omega(x)$ $\mu$-almost everywhere. Since
$\widetilde M\omega(x)$ is equivalent to $M\omega(x)$, it is
clear that
$$\frac 1{\mu(B)}\int_B \omega(y)\,d\mu(y) \le C\omega(x)\qquad \mu\text{-almost everywhere}\ x\in B.$$
Hence $\omega\in A_1$. Conversely, Theorem \ref{w1} shows that
there exists $C>0$ such that, for any $\lambda>0$ and $f\in
L^1_\omega$,
$$\int_{\{x\in X:\widetilde  M f(x)>\lambda\}}\omega(x) \, d\mu(x)\le \frac{C}{\lambda}\int_X |f(x)|\omega(x)\,d\mu(x).$$
Suppose $x\in B_1\subset B_2$. Let $f=\chi_{B_1}$ and $z\in
B_2$. Then
$$\widetilde Mf(z)\ge \frac 1{\mu(B_2)}\int_{B_2}f(y)\,d\mu(y)=\frac {\mu(B_1)}{\mu(B_2)}.$$
The above inequality shows that $B_2\subset \{x: \widetilde
Mf(x)\ge \mu(B_1)/\mu(B_2)\}$. Hence,
\begin{align*}
    \int_{B_2} \omega(x)\,d\mu(x)
    &\le \int_{ \{x: \widetilde Mf(x)\ge \mu(B_1)/\mu(B_2)\}} \omega(x)\,d\mu(x) \\
    &\le C \frac {\mu(B_2)}{\mu(B_1)} \int_{B_1}\omega(x)\,d\mu(x).
\end{align*}
By Lebesgue's differentiation theorem, the Lemma follows.
\end{proof}

We will need the following generalization to spaces of
homogeneous type of one direction of a well-known result of
Coifman and Rochberg~\cite{CR}.

\begin{lem}\label{max}
    Let $f\in L^1_{\text{\rm loc}}(X)$ such that $Mf(x)<\infty$ $\mu$-almost everywhere.
    Then $\big(Mf\big)^\delta\in A_1$ for $0\le \delta<1$.
\end{lem}

\begin{proof}
By Lemma~\ref{a1}, it suffice to show that there exists a
constant $C$ such that, for any $B$ and $\mu$-almost every
$x\in B$,
$$\frac 1{\mu(B)}\int_B \big(\widetilde Mf\big)^\delta \, d\mu\le C \big(\widetilde Mf(x)\big)^\delta.$$
Fix $B=B(x_0,t_0)$ and decompose $f$ as $f=f_1+f_2$, where
$f_1=f\chi_{_{2B}}$ and $f_2=f\chi_{_{(2B)^c}}$ with
$2B=B(x_0,2t_0)$. Then $\widetilde Mf(y)\le \widetilde
Mf_1(y)+\widetilde Mf_2(y)$ and
$$\big(\widetilde Mf(y)\big)^\delta\le \big(\widetilde Mf_1(y)\big)^\delta
        + \big(\widetilde Mf_2(y)\big)^\delta\qquad\text{for}\ \ 0\le\delta<1.$$ Since
$\widetilde M$ is weak (1,1) with respect to the measure $\mu$,
Kolmogorov's inequality shows that
$$\frac 1{\mu(B)}\int_B \big(\widetilde Mf_1(y)\big)^\delta \, d\mu(y)\le \frac{C}{\mu(B)}\mu(B)^{1-\delta}\|f_1\|_{L^1}^\delta
    \le C\Big(\frac 1{\mu(B)}\int_{2B} f \, d\mu\Big)^\delta\le C \big(\widetilde M f(x)\big)^\delta.$$

Now we estimate $\widetilde Mf_2$. Given $y\in B$, for any
$B(y_0,R)$ that contains $y$, we have $B\subset
B(y_0,A_0^2\max\{t_0,R\})$. If $R<t_0$, we have $B(y_0,t_0)\cap
B(x_0,t_0)\neq \varnothing$ and hence $B(y_0,t_0) \subset
B(x_0,A_0^2 t_0)$ which gives $B(y_0,\frac{t_0}{2A_0^2})\subset
B(x_0, 2t_0)$. Then the inequality $\int_{B(y_0,R)} |f_2| \,
d\mu > 0$ implies $R>\frac{t_0}{2A_0^2}$ that concludes
$B\subset B(y_0,2A_0^4R)$ when $R<t_0$. It is clear that
$B\subset B(y_0,2A_0^4R)$ when $R\ge t_0$. Thus,
\[
    \frac 1{\mu(B(y_0,R))}\int_{B(y_0,R)} |f_2|
    \le \frac C{\mu(B(y_0, 2A_0^4R))}\int_{B(y_0, 2A_0^4R)}|f_2| \, d\mu
    \le C \widetilde Mf(x),
\]
so that $\widetilde Mf_2(y)\le C\widetilde Mf(x)$ for all $y\in
B$. Therefore,
\[
    \frac 1{\mu(B)}\int_B \big(\widetilde Mf_2(y)\big)^\delta \, d\mu(y)
    \le C \big(\widetilde Mf(x)\big)^\delta.
\]
This completes the proof.
\end{proof}

\begin{lem}\label{a2}
    If $\omega\in A_2(X)$, then $\log \omega\in \bmo(X)$.
\end{lem}

We omit the proof of Lemma~\ref{a2}, which echoes the Euclidean
version (see for example~\cite{D}).

We are ready to show the main result in the one-parameter case.
We follow the proof in~\cite{JJ}.

\begin{proof}[Proof of Theorem \ref{thm:1.1} for one parameter]
Since $H^1(X)$ is a subspace of $L^1(X)$, it follows from
Fatou's lemma that $f\in L^1(X)$. To show \eqref{eq:1.1} for
all $\phi\in\vmo(X)$, by density it suffices to consider
$\phi\in \mathop G\limits^{\circ}(\beta,\gamma)$. Fix
$\delta\in (0,\frac1{2A_0})$ and pick $\eta>0$ such that
$\eta\exp(\delta^{-1})\le \delta C_\mu^{\log_2 \delta}$ and
$\int_E |f| \, d\mu \le \delta$ whenever $\mu(E)\le
C\eta\exp(\delta^{-1})$. Now choose $k$ large enough so that
$$\mu(E_k):=\mu\big(\big\{x\in X : |f_k(x)-f(x)|>\eta\big\}\big)\le \eta.$$

We construct a bump function~$\tau(x)$ on~$X$, as follows.
Define
$$\tau(x):=\max\big\{0, 1+\delta\log (M \chi_{{}_{E_k}})(x)\big\}.$$
It is clear that $0\le \tau(x)\le 1$ and $\tau\equiv 1$
$\mu$-almost everywhere on $E_k$. Also, $\|\tau\|_{\bmo(X)}\le
2\delta\|\log (M\chi_{{}_{E_k}})^{1/2}\|_{\bmo(X)}\le C\delta$
due to Lemmas \ref{max} and \ref{a2}. By the weak $(1,1)$
estimate for $M$ with respect to $\mu$,
$$\mu\big(\text{supp}(\tau)\big)\le C\mu(E_k)\exp(\delta^{-1})\le C\eta\exp(\delta^{-1}).$$
Consequently,
$$\int_{\text{supp}(\tau)}|f| \, d\mu\le \delta.$$
We now write
\begin{align*}
\bigg|\int_X(f-f_k)\phi \, d\mu\bigg|
&\le \bigg|\int_X (f-f_k)\phi(1-\tau) \, d\mu\bigg|+\bigg|\int_X (f-f_k)\phi\tau \, d\mu\bigg|\\
&\le \eta\|\phi\|_{L^1(d\mu)}+\int_{\text{supp}(\tau)}|f| \, d\mu+\bigg|\int_X f_k\phi\tau \, d\mu\bigg|\\
&\le \delta+\delta+\bigg|\int_X f_k\phi\tau \, d\mu\bigg|.
\end{align*}
The proof of \eqref{eq:1.1} will therefore be established provided we verify
\begin{equation}\label{eq:3.4}
    \|\phi\tau\|_{\bmo(X)}\le C\delta.
\end{equation}
Suppose $B=B(y_0,r_0)$ with $r_0<\delta$. The H\"older
regularity of $\phi$ gives
\begin{align*}
    \frac1{\mu(B)}\int_B |\phi\tau-(\phi\tau)_B| \, d\mu
    &\le \frac2{\mu(B)}\int_B |\phi\tau-\phi_B\tau_B| \, d\mu\\
    &\le  \frac2{\mu(B)}\int_B |\phi\tau-\phi_B\tau| \, d\mu+\frac {2|\phi_B|}{\mu(B)}\int_B |\tau-\tau_B| \, d\mu\\
    &\le C\delta^\beta+2\|\phi\|_{L^\infty}\|\tau\|_{\bmo(X)}\\
    &\le C(\delta^\beta+\delta).
\end{align*}
For $r_0>\delta$ and $B(y_0,\delta)\cap
B(x_0,\delta^{-1})=\varnothing$, the size condition of $\phi$
yields
\begin{align*}
    \frac1{\mu(B)}\int_B |\phi\tau-(\phi\tau)_B| \, d\mu
    &\le \frac2{\mu(B)}\int_B |\phi\tau| \, d\mu\\
    &\le  C\delta^\gamma.
\end{align*}
For $r_0>\delta$ and $B(y_0,\delta)\cap
B(x_0,\delta^{-1})\ne\varnothing$, we obtain
$B(y_0,\delta^{-1})\subset B(x_0,A_0 \delta^{-1})$ and hence
$\mu(B(x_0,\delta^{-1}))\le \mu(B(y_0,A_0 \delta^{-1}))$. The
doubling condition shows that $$\mu(B(y_0,A_0\delta^{-1}))\le
C_\mu^{\log_2(A_0\delta^{-2})}\mu(B(y_0,\delta)).$$ Thus,
$$\frac 1{\mu(B)}\le \frac{C_\mu^{\log_2(A_0\delta^{-2})}}{\mu(B(y_0,A_0\delta^{-1}))}
\le \frac{C_\mu^{\log_2(A_0\delta^{-2})}}{\mu(B(x_0,\delta^{-1}))}
\le \frac {C_\mu^{\log_2(A_0\delta^{-2})}}{V_1(x_0)},$$
and then
\begin{align*}
    \frac1{\mu(B)}\int_B |\phi\tau-(\phi\tau)_B| \, d\mu
    &\le \frac2{\mu(B)}\int_B |\phi\tau| \, d\mu\\
    &\le  \frac {2C_\mu^{\log_2(A_0\delta^{-2})}}{V_1(x_0)}\mu(\text{supp}(\tau))\\
    &\le \frac {2C_\mu^{\log_2(A_0\delta^{-2})}}{V_1(x_0)}\eta\exp(\delta^{-1})\le C\delta.
\end{align*}
Therefore,
\begin{align}\label{eqn:3star}
    \frac1{\mu(B)}\int_B |\phi\tau-(\phi\tau)_B| \, d\mu
    \le C\delta.
\end{align}
and \eqref{eq:3.4} follows. By weak-star compactness of the
ball in $H^1$, there exists a $g\in H^1$ with $\|g\|_{H^1}\le1$
and a subsequence $\{f_{k_l}\}_{l\in \Bbb N}$ such that
$\{f_{k_l}\}_{l\in \Bbb N}$ weak-star converges to  $g$. By
\eqref{eq:1.1}, we have $\int f\phi=\int g\phi$ for all
$\phi\in \mathop G\limits^{\circ}(\beta,\gamma)$, and hence
$f=g\in H^1$.
\end{proof}

\section{Proof of Theorem \ref{thm:1.1} in the product case}\label{sec-thm1.1-product}
\setcounter{equation}{0}

We begin by recalling several key tools we will use to pass
from the product Euclidean setting to the setting of product
spaces of homogeneous type. These tools are the random dyadic
lattices, the dyadic product $\bmo$ space, the averaging
theorem relating the dyadic and continuous product $\bmo$
spaces, several properties of product $\littlebmo$ (``little
$\bmo$''), and the construction of a product bump
function~$\tau(x_1,x_2)$ on $X_1\times X_2$. Then we prove
Theorem~\ref{thm:1.1} for product spaces of homogeneous type.

In~\cite[Theorem 5.1]{HK} Hyt\"onen and Kairema constructed
random dyadic lattices on spaces of homogeneous type, extending
an earlier result of Nazarov, Treil and Volberg~\cite{NTV}.
Specifically, there exists a probability space $(\Omega,\PP)$
such that for each $\omega\in\Omega$ there is an associated
dyadic lattice $\D(\omega) = \{Q_\alpha^k(\omega)\}_{k,\alpha}$
related to dyadic points $\{x_\alpha^k(\omega)\}_{k,\alpha}$,
with the properties \eqref{DyadicP1}--\eqref{DyadicP5} above,
and the following \emph{smallness property} holds: there exist
absolute constants $C$, $\eta > 0$ such that
\begin{eqnarray}\label {Rcondition}
    \PP\left(\{\omega\in\Omega: x,x^*
        \mbox{ are not in the same cube } Q\in\D^k(\omega) \}\right)
    \leq C\left(\frac{\rho(x,x^*)} {\delta^k}\right)^\eta
\end{eqnarray}
for all $x,x^*\in X$, where $\D^k(\omega)$ is the set of all
dyadic cubes at level $k$ in $\D(\omega)$.

Fix $\om\in\Om$. For a cube $Q\in\mathcal{D}(\om)$, let
$\textup {ch}(Q)$ denote the set of all children of $Q\in
\D(\omega)$. From \eqref{DyadicP1} and \eqref{DyadicP2}, we
know that~$Q=\bigcup_{I\in \textup {ch}(Q)}I$. For a cube $Q\in
\D(\omega)$, define the {\emph{averaging
operator}}~$E_Q^{\omega}$ by
$$
E_Q^\omega f=E^{\D(\omega)}_Qf:=\Big(\intav_Q f\,d\mu\Big)\chi_Q,
$$
where as usual ${-\!\!\!\!\!\int}_Q
f\,d\mu=\mu(Q)^{-1}\int_Qf\,d\mu$ and $\chi_Q$ is the
characteristic function of $Q$. (We reserve the more usual name
of \emph{expectation operator} for the expectation
$\mathbb{E}_\omega$ over random dyadic lattices, defined
below.) Define the {\emph{difference operator}}
$\Delta_Q^\omega$ by
$$
    \Delta_Q^\omega f=\Delta^{\D(\omega)}_Q f
    :=\Big(\sum_{J\in \textup{ch}(Q)}E^\omega_Jf\Big)-E^\omega_Qf.
$$
For convenience, we sometimes write  $E_Q$ and $\Delta_Q$
instead of $E^{\omega}_Q$ and $\Delta^{\omega}_Q$. Note that
for every $x\in X$, at each level $k$ there exists exactly one
cube $Q^k(x)\in \D^k(\omega)$ such that $x\in Q^k(x)$. So for
each $k\in \Z$ we can define
\[
    E_kf(x)
    := \sum_{\alpha}E_{Q^k_\alpha}f(x)
    = E_{Q^k(x)}f(x)
    \textup { and }
    \Delta_kf(x)
    := \sum_{\alpha}\Delta_{Q^k_\alpha}f(x)
    = E_{k+1}f(x) - E_kf(x).
\]

For $j=1,2$, let $(\Omega_j,\PP_j)$ be a probability space for
$(X_j,\rho_j,\mu_j)$ such that for each $\omega_j\in\Omega_j$
there is an associated dyadic lattice $\mathcal{D}_j(\omega_j)$
satisfying properties (1)--(6). We define the  dyadic product
$\bmo(X_1\times X_2)$ space via the difference operator.
Let~$\Delta^{\omega} := \Delta_{Q_1^{\omega_1}}^{\omega_1}
\Delta_{Q_2^{\omega_2}}^{\omega_2}$ where $Q_1^{\omega_1}\in
\D_1(\omega_1)$ and $Q_2^{\omega_2}\in \D_2(\omega_2)$. Let
$R^\omega$ denote the rectangle $Q_1^{\omega_1}\times
Q_{2}^{\omega_2}$.

\begin{dfn}
    Let $f^\omega(x)=f^{(\omega_1,\omega_2)}(x_1,x_2)$ be a
    locally integrable function on $X_1\times X_2$. We say that $f^\omega$
    belongs to the {\emph{dyadic product bounded mean oscillation
    space}} $\bmo_{\omega_1,\omega_2} :=\break
    \bmo_{\D_1(\omega_1)\times \D_2(\omega_2)}(X_1\times X_2)$ if there
    exists a constant $C>0$ such that for every open set
    $\mathcal{A}\subset X_1\times X_2$,
    \[
        \frac{1}{\mu(\mathcal{A})}
            \sum_{R^\omega\subset \mathcal{A}}\int_{\widetilde{X}}
            |\Delta^\omega f^\omega|^2\,d\mu
        \leq C^2.
    \]
    We define the dyadic product $\bmo$ norm
    $\|f^\omega\|_{\bmo_{\omega_1,\omega_2}}$ of the function
    $f^\omega$ to be the infimum of $C$ such that the inequality
    above holds.
\end{dfn}

\begin{thm}[\cite{CLW}]\label{thm:Bi_BMO}
    Let $(X_1,d_1,\mu_1)$ and $(X_2,d_2,\mu_2)$ be spaces
    of homogeneous type. For $j =
    1$, $2$, let $(\Omega_j,\mathbb{P}_j)$ be a probability space, and
    $\{\D(\om_j)\}_{\om_j\in\Omega_j}$ a collection of random
    dyadic lattices on~$X_j$, such that
    properties~\textup{(1)--(6)} hold. Let $\{f^\omega\}$,
    $\omega := (\omega_1,\omega_2)\in \Omega_1\times \Omega_2$,
    be a family of functions with $f^\omega\in
    \bmo_{\mathcal{D}(\omega_1)\times\mathcal{D}(\omega_2)}(X_1\times X_2)$
    for each $\omega\in\Omega_1\times\Omega_2$, such that
    \begin{enumerate}
            \item[(i)] $\omega\mapsto f^\omega$ is
                measurable, and

            \item[(ii)]
                $\|f^\omega\|_{\bmo_{\mathcal{D}(\omega_1)
                \times\mathcal{D}(\omega_2)}(X_1\times
                X_2)}\leq C_d$ for some constant $C_d$
                independent of $\omega$.
        \end{enumerate}
    Then the function $f$ defined by the expectation
        \[
            f(x):=\mathbb{E}_{\omega}f^{\omega}(x)
        \]
    belongs to  $\bmo(X_1\times X_2)$, and $\|\mathbb{E}_\omega
    f^\omega\|_{\bmo(X_1\times X_2)}\leq CC_d$.
\end{thm}

\begin{dfn}
    A real-valued function $f\in L^1_{\text{loc}}(X_1\times X_2)$ is in
    the space $\littlebmo(X_1\times X_2)$ (called ``little $\bmo$'' in the
    literature) if its $bmo$ norm is finite:
    \begin{align}
        \|f\|_{\littlebmo(X_1\times X_2)}
        := \sup_R \int_R |f(x_1,x_2)-f_R| \, d\mu_1(x_1) \, d\mu_2(x_2)
        < \infty.
    \end{align}
\end{dfn}

\begin{lem}\label{lem 4.3}
    If $f$ and $g$ belong to $\littlebmo$, then $\max\{f,g\}\in \littlebmo$.
\end{lem}

\begin{lem}\label{lem 4.1}
    Suppose $\Omega$ is an open set in $X_1\times X_2$ with
    finite measure. Let $\mathcal{D}_1$ and $\mathcal{D}_2$ be
    dyadic cubes in $X_1$ and $X_2$, respectively. Then
    \[
        \sum_{R=Q_1\times Q_2\in \mathcal{D}_1\times\mathcal{D}_2, R\subset \Omega}
            \|\Delta_{Q_1\times Q_2} f\|_2^2
        \le \int_{X_1} \sum_{Q_2\in \mathcal{D}_2(\om_2)}
            \|\Delta_{Q_2} f(x_1,\cdot)\|_2^2 \, d\mu_1(x_1).
    \]
\end{lem}

\begin{proof}
Let $\tilde f(x_1,\cdot)=\sum_{Q_2\in
\mathcal{D}_2(\om_2)}\Delta_{Q_2} f(x_1,\cdot)$ for $x_1\in
X_1$. Then $\Delta_{Q_2} f(x_1,\cdot)=\Delta_{Q_2} \tilde
f(x_1,\cdot)$. Since $\Delta_{Q_1\times Q_2}=\Delta_{Q_2}
\otimes \Delta_{Q_2}$, we get that
$$\Delta_{Q_1\times Q_2} f=\Delta_{Q_1\times Q_2} \tilde f,$$
and so
\begin{align*}
    \sum \|\Delta_{Q_1\times Q_2} f\|_2^2
    &=\sum \|\Delta_{Q_1\times Q_2} \tilde f\|_2^2 \\
    &\le \|\tilde f\|_2^2=\int_{X_1}
        \bigg\|\sum_{Q_2\in \mathcal{D}_2(\om_2)}\Delta_{Q_2} f(x_1,\cdot) \bigg\|_2^2
        \, d\mu_1(x_1)\\
    &=\int_{X_1} \sum_{Q_2\in \mathcal{D}_2(\om_2)}\|\Delta_{Q_2} f(x_1,\cdot)\|_2^2
        \, d\mu_1(x_1).
    \qedhere
\end{align*}
\end{proof}

\begin{lem}\label{lem 4.2}
    Suppose $\phi\in \mathop G\limits^{\circ}(\beta_1,\beta_2,
    \gamma_1, \gamma_2)$ and $b$ is a bounded function with
    $\|b\|_\infty\le 1$. Then, for all $\alpha\in(0,1)$, for each
    open $\Omega\subset X_1\times X_2$, and for each rectangle
    $R=Q_1 \times Q_2 \in \mathcal{D}_1(\om_1) \otimes
    \mathcal{D}_2(\om_2)$, we have
    \begin{equation}\label{eq 4.1}
        \sum_{R\subset \Omega, \, \diam(R) \le \alpha} \|\Delta_R (\phi b)\|_2^2
        \le C(\|b\|_{\littlebmo}+\alpha) \, \mu(\Omega).
    \end{equation}
\end{lem}

\begin{proof}
The proof is by iteration. For one parameter, it suffices to
prove \eqref{eq 4.1} for $\Omega=Q_0$, where $Q_0$ is a dyadic
cube in~$X_1$. Without loss of generality we may assume that
$\diam Q_0\le \alpha$. Then
\begin{align*}
    \sum_{Q\subset Q_0} \|\Delta_Q (\phi b)\|_2^2
    &=\int_{Q_0} |\phi b(x)-(\phi b)_{Q_0}|^2 \, d\mu(x) \\
    &\le 2 \int_{Q_0} |\phi b(x)-(\phi)_{Q_0}(b)_{Q_0}|^2 \, d\mu(x)\\
    &\le 2 \int_{Q_0} |\phi b(x)-\phi(x)(b)_{Q_0}|^2 \, d\mu(x)
        + 2 \int_{Q_0} |\phi(x) (b)_{Q_0}-(\phi)_{Q_0}(b)_{Q_0}|^2 \, d\mu(x)\\
    &\le C(\|b\|^2_{bmo}+\alpha) \, \mu(\Omega)
\end{align*}
by inequality~\eqref{eqn:3star}. Applying Lemma~\ref{lem 4.1},
we obtain
\begin{align*}
    \sum_{Q_1\in \mathcal{D}_1(\om_1), Q_2\in \mathcal{D}_2(\om_2)} \|\Delta_{Q_1\times Q_2} f\|^2
    &\le \int_{X_1} \sum_{Q_2\in \mathcal{D}_2(\om_2)} \|\Delta_{Q_2} f(x_1,\cdot)\|^2 \, d\mu_1(x_1) \\
    &\qquad+ \int_{X_2} \sum_{Q_1\in \mathcal{D}_1(\om_1)} \|\Delta_{Q_1} f(\cdot, x_2)\|^2 \, d\mu_2(x_2) \\
    &\le C(\|b\|^2_{bmo}+\alpha)\int_{X_1} \mu_2(\{x_2: (x_1,x_2)\in \Omega\}) \, d\mu_1(x_1) \\
    &\qquad +C(\|b\|^2_{bmo}+\alpha)\int_{X_2} \mu_1(\{x_2: (x_1,x_2)\in \Omega\}) \, d\mu_2(x_2)\\
    &\le 2C(\|b\|^2_{bmo}+\alpha) \, \mu(\Omega).
    \qedhere
\end{align*}
\end{proof}

Next we construct a bump function $\tau(x_1,x_2)$ in the
product setting.

\begin{lem}\label{lem 4.4}
    Let $E $ be a subset of $X_1\times  X_2$  with finite
    measure, and let $\delta\in (0,1)$ be a given parameter.
    Then there exists a function $\tau\in \littlebmo$ such that
    $\tau\equiv 1$ on $E$, $\|\tau\|_{bmo}<C_1\delta$, and $\mu($supp
    $\tau)<C_2e^{2/\delta}\mu(E)$, where $C_1$ and $C_2$ are
    some absolute constants.
\end{lem}

\begin{proof}
Let $M_s$ be the strong maximal function, in which the averages
are taken over arbitrary rectangles in $X_1\times X_2$. A
weight $w$ is in $A_1(X_1 \times X_2)$ if there exists a
constant $C$ such that $M_s w(x)\le Cw(x)$ for $\mu$-almost
every $x\in X_1\times X_2$.

We define the following $A_1$ weight, with $M^{(k)}_s$ denoting
the $k$-fold iteration of the strong maximal function:
\[
    m(x_1,x_2)
    = K^{-1}\sum_{k=0}^\infty c^kM^{(k)}_s \chi_E(x_1,x_2),
\]
where $K=\sum_k c^k$ and $c>0$ is chosen to insure the
convergence of the series. Then $\|m\|_2\le C\|\chi_E\|_2 =
C\mu(E)^{1/2}.$ Observe that $m=1$ $\mu$-almost everywhere
on~$E$, and $m\le 1$ $\mu$-almost everywhere outside~$E$.

Define the function
\[
    \tau(x_1,x_2)
    := \max\{0,1+\delta \log m(x_1,x_2)\}.
\]
The function $\tau$ is in $\littlebmo$, and satisfies $\tau=1$
$\mu$-almost everywhere on~$E$. By Lemma~\ref{lem 4.3} and the
fact that $\log w\in bmo$ for every $A_1$ weight~$w$, which is
proved exactly as in the one-parameter Euclidean setting, we
have $\|\tau\|_{bmo}\le C\delta$.

The estimate for the size of the support of $\tau$ follows from
Tchebychev's theorem and the estimate $\|m\|_2\le
C\mu(E)^{1/2}$.
\end{proof}

We are ready to prove our main result for product spaces of
homogeneous type. We follow the lines of the product Euclidean
proof from~\cite{PT}.

\begin{proof}[Proof of Theorem~\ref{thm:1.1} in the product case]
First note that $ \mathop
G\limits^{\circ}(\beta_1,\beta_2,\gamma_1,\gamma_2) $ is dense
in \vmo$(X_1\times X_2)$. To prove Theorem  \ref{thm:1.1}, it
suffices to show \eqref{eq:1.1} for all $\phi\in \mathop
G\limits^{\circ}(\beta_1,\beta_2,\gamma_1,\gamma_2)$.

Next, note that as shown in \cite{HLPW}, $H^1(X_1\times X_2)$
is a subspace of $L^1(X_1\times X_2)$. Thus, since $
f_n\rightarrow f $ a.e, and $\|f_n\|_{H^1(X_1\times X_2)}\leq
1$, by Fatou's lemma we have  that $f\in L^1(X_1\times X_2)$
with $\|f\|_{L^1(X_1\times X_2)}\leq 1$.

Fix $\delta\in (0,\frac1{2A_0})$ and pick $\eta>0$ such that
$\eta\exp(2/\delta)\le \delta C_\mu^{\log_2 \delta}$ and
$\int_E |f| \, d\mu \le \delta$ whenever $\mu(E)\le
C_2\eta\exp(2/\delta)$, where $C_2$ is as in Lemma~\ref{lem
4.4}. Now choose $K_0$ large enough such that when $k>K_0$,
\[
    \mu(E_k)
    :=\mu\big(\big\{(x_1,x_2)\in X_1\times X_2 :
        |f_k(x_1,x_2)-f(x_1,x_2)|>\eta\big\}\big)
    \le \eta.
\]
Define
$$\tau(x_1,x_2)=\max\big\{0, 1+\delta\log m(x_1,x_2)\big\},$$
where $m(x_1,x_2)=K^{-1}\sum_{\ell=0}^\infty c^\ell
M^{(\ell)}_s \chi_{E_k}(x_1,x_2)$ as defined in Lemma~\ref{lem
4.4}. It is clear that $0\le \tau(x_1,x_2)\le 1$ and $\tau=1$
$\mu$-almost everywhere on $E_k$. By Lemma~\ref{lem 4.4}, we
have $\tau\in bmo$ with  $\|\tau\|_{bmo}\le C_2\delta$ and
$$\int_{\text{supp}(\tau)}|f| \, d\mu \le \delta.$$
For every $k> K_0$, we now write
\begin{align*}
    \int_{X_1\times X_2}(f-f_k)\phi \, d\mu
    =\int_{X_1\times X_2} (f-f_k)\phi(1-\tau) \, d\mu
        +\int_{X_1\times X_2} (f-f_k)\phi\tau \, d\mu.
\end{align*}
Note that $\tau=1$ $\mu$-almost everywhere on $E_k$. In the
complement of $E_k$ we have $|f-f_k|<\eta$. Thus the first
interval on the right-hand side of the above equality is
bounded by $\eta\|\phi\|_{L^1(X_1\times X_2)}$, which is in
turn less than $\delta$ if $\eta$ is sufficiently small.
Further, the second interval is bounded by
\begin{align*}
    &\int_{\text{supp}(\tau)}|f\phi| \, d\mu
        + \bigg|\int_{X_1\times X_2} f_k\phi\tau \, d\mu\bigg|\\
    &\le \delta+\bigg|\int_{X_1\times X_2} f_k\phi\tau \, d\mu\bigg|.
\end{align*}

The proof of \eqref{eq:1.1} will therefore be established provided we verify
\begin{equation}\label{eq:4.7}
    \|\phi\tau\|_{BMO(X_1\times X_2)}\le C\delta.
\end{equation}

We will now verify \eqref{eq:4.7} by first proving that the
dyadic BMO norm of $\phi\tau$ has the required estimate, and
then by using Theorem \ref{thm:Bi_BMO}.

For every arbitrary open set $\mathcal{A}\subset X_1\times X_2$
with finite measure and $x\in \mathcal{A}$, there exists a
constant $r(x)<\frac \delta{3A_0}$ such that $B(x,r(x))\subset
\mathcal{A}$ and then
\[
    \mathcal{A}
    = \bigcup\limits_{x\in\mathcal{A}} B(x,r(x)).
\]
By \cite[Lemma 3]{C1}, there exists a
countable subfamily of disjoint spheres $B(x_i,r(x_i))$ such
that each sphere $B(x,r(x)), x\in \mathcal{A}$ is contained in
$B(x_i,3A_0r(x_i))$ for some $i\in \Bbb N$. Hence,
\[
    \int_{\mathcal{A}} |\phi\tau|^2 \, d\mu
    \le \sum_{i=1}^\infty \int_{B(x_i,3A_0r(x_i))} |\phi\tau|^2 \, d\mu.
\]
Since $3A_0r(x_i)<\delta$, we use Lemma \ref{lem 4.2} to get
\[
    \int_{B(x_i,3A_0r(x_i))} |\phi\tau|^2 \, d\mu
    = \sum_{R\subset B(x_i,3A_0r(x_i))} ||\Delta_R (\phi b)\|_2^2
    \le C(\|\tau\|_{bmo}+\delta)\mu(B(x_i,3A_0r(x_i))).
\]
Therefore,
\[
    \int_{\mathcal{A}} |\phi\tau|^2 \, d\mu(x)
    \le \sum_{i=1}^\infty \int_{B(x_i,3A_0r(x_i))} |\phi\tau|^2 \, d\mu(x)
    \le C\delta \sum_i \mu(B(x_i,3A_0r(x_i))).
\]
Since $\mu(B(x_i,3A_0r(x_i)))\le C\mu(B(x_i,r(x_i)))$ and
$\{B(x_i,r(x_i))\}_{i\in \Bbb N}$ are disjoint, we have
\[
    \int_{\mathcal{A}} |\phi\tau|^2 \, d\mu(x)
    \le\sum_i \mu(B(x_i,3A_0r(x_i)))\le C\mu(\mathcal{A}).
\]
This completes the proof of~\eqref{eq:4.7}.
\end{proof}


\end{document}